\documentclass[12pt]{amsart}
\usepackage{amssymb}
\usepackage{geometry}
\usepackage{amsmath, amsthm, amssymb,color}
\newtheorem{thm}{Theorem}[section]

\newtheorem{prop}[thm]{Proposition}
\newtheorem{lem}[thm]{Lemma}
\newtheorem{example}[thm]{Example}
\newtheorem{defn}[thm]{Definition}

\newtheorem{conj}[thm]{Conjecture}
\newtheorem{conv}[thm]{Convention}

\def \a {\alpha}
\def \G {W(\X)}
\def \ha {\widehat{\alpha}}
\def \wh {\widehat}
\def \X {\bold{X}}
\def \Y {\bold{Y}}

\title{Weak ergodic averages over dilated measures}

\author{Wenbo Sun}
\address{Department of Mathematics, The Ohio State University, 231 West 18th Avenue, Columbus OH, 43210-1174, USA}
\email{sun.1991@osu.edu}

\begin{document}
	
\maketitle	
	
\begin{abstract}
	Let $m\in\mathbb{N}$ and $\X=(X,\mathcal{X},\mu,(T_{\alpha})_{\alpha\in\mathbb{R}^{m}})$ be a measure preserving system with an $\mathbb{R}^{m}$-action.  We say that a Borel measure $\nu$ on $\mathbb{R}^{m}$ is weakly equidistributed for $\X$ if there exists $A\subseteq\mathbb{R}$ of density 1 such that for all $f\in L^{\infty}(\mu)$, we have
	$$\lim_{t\in A,t\to\infty}\int_{\mathbb{R}^{m}}f(T_{t \alpha}x)\,d\nu(\alpha)=\int_{X}f\,d\mu$$
	for $\mu$-a.e. $x\in X$. 
	
	Let $\G$ denote the collection of all $\a\in\mathbb{R}^{m}$ such that the $\mathbb{R}$-action $(T_{t\a})_{t\in\mathbb{R}}$ is not ergodic.  Under the assumption of the pointwise convergence of double Birkhoff ergodic average, we show that a Borel measure $\nu$ on $\mathbb{R}^{m}$ is weakly equidistributed for an ergodic system $\X$ if and only if $\nu(\G+\beta)=0$ for every $\beta\in\mathbb{R}^{m}$. Under the same assumption, we also show that $\nu$ is weakly equidistributed for all ergodic measure preserving systems with $\mathbb{R}^{m}$-actions if and only if $\nu(\ell)=0$ for all hyperplanes $\ell$ of $\mathbb{R}^{m}$. 
	
	Unlike many equidistribution results in literature whose proofs use methods from harmonic analysis, our results adopt a purely ergodic theoretic approach.
\end{abstract}	
	
\section{introduction}
\subsection{Strong equidistribution over dilated measures}
Let $G$ be a locally compact Hausdorff topological group. A \emph{measure preserving $G$-system} (or a \emph{$G$-system}) is a tuple $\X=(X,\mathcal{X},\mu,(T_{g})_{g\in G})$, where $(X,\mathcal{X},\mu)$ is a separable probability space and $T_{g}\colon X\to X, g\in G$ are measurable and  measure preserving transformations such that $T_{g}\circ T_{h}=T_{gh}$, $T_{e_{G}}=id$ for all $g,h\in G$. We also require that for all $x\in X$, the map $G\to X, g\to T_{g}x$ is measurable. 
We say that $\X$ is \emph{ergodic} if $A\in\mathcal{X}, T_{g}A=A$ for all $g\in G$ implies that $\mu(A)=0$ or 1. 

Let $m\in\mathbb{N}$, $\nu$ be a Borel measure on $\mathbb{R}^{m}$ and $\X=(X,\mathcal{X},\mu,(T_{g})_{g\in \mathbb{R}^{m}})$ be an ergodic $\mathbb{R}^{m}$-system.\footnote{Throughout this paper, we assume that $\nu(\mathbb{R}^{m})=1$.} We say that $\nu$ is \emph{(strongly) equidistributed} for $\X$ if for all $f\in L^{\infty}(\mu)$ %(i.e. $f$ is a measurable function on $(X,\mathcal{X},\mu)$ with a bounded $L^{\infty}$ $\mu$-norm) 
we have 
$$\lim_{t\to\infty}\int_{\mathbb{R}^{m}}f(T_{t \alpha}x)\,d\nu(\alpha)=\int_{X}f\,d\mu$$
for $\mu$-a.e. $x\in X$.
  Birkhoff ergodic theorem for $\mathbb{R}$-systems (see for example Corollary 8.15 \cite{ET}) states that for every ergodic $\mathbb{R}$-system $(X,\mathcal{X},\mu,(T_{t})_{t\in \mathbb{R}})$ and every $f\in L^{\infty}(\mu)$, $$\lim_{T\to\infty}\frac{1}{T}\int_{0}^{T}f(T_{t }x)=\int_{X}f\,d\mu$$
  for $\mu$-a.e. $x\in X$, which
  is equivalent to say that the Lebesgue measure restricted to the interval $[-1,1]$ is equidistributed for every ergodic $\mathbb{R}$-system. Similarly, Birkhoff ergodic theorem holds for  $\mathbb{R}^{m}$-system for all $m\in\mathbb{N}$: the Lebesgue measure restricted to the unit cube or ball in $\mathbb{R}^{m}$ is equidistributed for every ergodic $\mathbb{R}^{m}$-system (see for example Theorem 8.19 \cite{ET}).

It is an interesting to ask if  similar results hold for the Lebesgue measure restricted to the boundary of the unit cube or ball. 
The motivation of this question was from a result of Stein \cite{St}
in 1976, who showed that for $\phi\in L^{p}(\mathbb{R}^{m}), p>m/(m-1), m\geq 3$, and for Lebesgue-a.e. $x\in\mathbb{R}^{m}$, we have that
$$\lim_{t\to 0}\int_{S_{t}}\phi(x-u)\,d\sigma_{m,t}(u)=\phi(x),$$ 
where $\sigma_{m,t}$ is the Lebesgue measure on $\mathbb{R}^{m}$ restricted to $S_{t}, $ the sphere of radius $t$ centered at the origin. Later an analog of this result was proved in the ergodic theoretic setting. It was proved by Jones \cite{J} (for $m\geq 3$) and Lacey \cite{dis2} (for $m=2$) that the Lebesgue measure restricted to the boundary of the unit ball $\sigma_{m,1}$ is equidistributed for all ergodic $\mathbb{R}^{m}$-systems. We remark that on the other hand, it is not hard to see that the Lebesgue measure restricted to the boundary of the unit cube is not equidistributed for some ergodic $\mathbb{R}^{m}$-systems.
It is then natural to ask which measure $\nu$ on $\mathbb{R}^{m}$ is  equidistributed for all ergodic $\mathbb{R}^{m}$-systems. It was proved by Bj\"orklund \cite{Bj} that if $\nu$ has \emph{Fourier dimension} $a>1$, meaning that $a$ is the supremium over all $0\leq a\leq d$ such that $\lim_{\zeta\to\infty}\vert\widehat{\nu}(\zeta)\vert\cdot\vert\zeta\vert^{a/2}<\infty,$ then  $\nu$  is  equidistributed for all ergodic $\mathbb{R}^{m}$-systems. It is worth noting that strong equidistribution for polynomial maps on special homogeneous systems have also been studied in recent years (see \cite{KSS,SY} for example).

\subsection{Weak equidistribution over dilated measures}
In contrast to the strong equidistribution, a notion called
``weakly equidistribution" were studied recently, and various results were obtained in the settings of translation surfaces \cite{CH} and nil manifolds \cite{KSS}. 
To be more precise,  we say that a Borel measure $\nu$ on $\mathbb{R}^{m}$ is \emph{weakly equidistributed} for  an  $\mathbb{R}^{m}$-system $\X$ if there exists $A\subseteq\mathbb{R}$ of density 1 such that for all $f\in L^{\infty}(\mu)$, we have
$$\lim_{t\in A,t\to\infty}\int_{\mathbb{R}^{m}}f(T_{t \alpha}x)\,d\nu(\alpha)=\int_{X}f\,d\mu$$
for $\mu$-a.e. $x\in X$.   
By the examples in Section 5 of \cite{KSS}, strong and weak equidistributions are not equivalent conditions. 

It is natural to ask which measures $\nu$ on $\mathbb{R}^{m}$ are weakly equidistributed for all ergodic $\mathbb{R}^{m}$-systems.
In this paper, we provide a necessary and sufficient condition for such $\nu$ 
 under the assumption of the pointwise convergence of double Birkhoff ergodic average. We say that an $\mathbb{R}^{m}$-system $(X,\mathcal{X},\mu,(T_{g})_{g\in \mathbb{R}^{m}})$ is \emph{good for double Birkhoff averages} if for all $f_{1},f_{2}\in L^{\infty}(\mu)$ and $\alpha_{1},\alpha_{2}\in\mathbb{R}^{m}$, the limit
$$\lim_{T\to\infty}\frac{1}{T}\int_{0}^{T}f_{1}(T_{\alpha_{1}t}x)f_{2}(T_{\alpha_{2}t}x)\,dt$$ 
exists for $\mu$-a.e. $x\in X$.

Our first theorem is the following:
\begin{thm}\label{main}
	Let $m\in\mathbb{N}$. A Borel measure $\nu$ on $\mathbb{R}^{m}$ is weakly equidistributed for all ergodic $\mathbb{R}^{m}$-systems which are good for double Birkhoff averages if and only if $\nu(\ell)=0$ for all hyperplanes\footnote{A \emph{hyperplane} of $\mathbb{R}^{m}$ is $V+\beta$ for some subspace $V$ of $\mathbb{R}^{m}$ of co dimension 1 and $\beta\in\mathbb{R}^{m}$.} $\ell$ of $\mathbb{R}^{m}$.
\end{thm}

Theorem \ref{main} will provide a complete answer for the weak equidistribution problem if the following conjecture holds:
\begin{conj}\label{main2}
	Every  $\mathbb{R}^{m}$-system is good for double Birkhoff averages.
\end{conj}
Conjecture \ref{main2} is still an open question in ergodic theory. Nevertheless,
various partial results on Conjecture \ref{main2} were obtain in recent years for some special type of systems, and so this paper can be viewed as an application of these results. We defer the discussion of this topic to Section \ref{s21}.

Another question we study in this paper is the necessary and sufficient conditions for a Borel measure $\nu$ on  $\mathbb{R}^{m}$ to be weakly equidistributed on a particular $\mathbb{R}^{m}$-system.
Let $m\in\mathbb{N}$ and $\X=(X,\mathcal{X},\mu,(T_{g})_{g\in\mathbb{R}^{m}})$ be an $\mathbb{R}^{m}$-system. Let $\G$ denote the collection of all $\a\in\mathbb{R}^{m}$ such that $I((t\a)_{t\in\mathbb{R}})\neq I(\mathbb{R}^{m})$.\footnote{$I(H)$ denote the $\sigma$-algebra of $\mathcal{X}$ consisting of all the $H$-invariant sets for every subgroup $H$ of $\mathbb{R}^{m}$.}
If $\X$ is an ergodic $\mathbb{R}^{m}$-system, then $\G$ is the collection of all $\a\in\mathbb{R}^{m}$ such that the $\mathbb{R}$-action $(T_{t\a})_{t\in\mathbb{R}}$ is not ergodic on $\X$. 
We have the following result:

\begin{thm}\label{main3}
	Let $m\in\mathbb{N}$ and $\nu$ be a Borel measure on $\mathbb{R}^{m}$. If $\X$ is an ergodic $\mathbb{R}^{m}$-system which is good for double Birkhoff averages such that  $\nu(\G+\beta)=0$ for every $\beta\in\mathbb{R}^{m}$, then $\nu$ is weakly equidistributed for $\X$. Conversely, if $\X$ is an ergodic $\mathbb{R}^{m}$-system such that  $\nu(\G+\beta)\neq 0$ for some $\beta\in\mathbb{R}^{m}$, then $\nu$ is not weakly equidistributed for $\X$.
\end{thm}
We remark that the second part of Theorem \ref{main3} holds for every ergodic $\mathbb{R}^{m}$-system. 
We give an example to illustrate Theorem \ref{main3}.
\begin{example}
	Let $m\in\mathbb{N}$ and $(X=\mathbb{T}^{m},\mathcal{X},\mu)$ be an $m$-dimensional torus endowed with the Lebesgue measure $\mu$. For all $\a\in\mathbb{R}^{m}$, denote $T_{\a}\beta=\a+\beta\mod\mathbb{Z}^{m}$ for all $\beta\in\mathbb{T}^{m}$. Then $\X=(X,\mathcal{X},\mu,(T_{\a})_{\a\in\mathbb{R}^{m}})$ is an ergodic  $\mathbb{R}^{m}$-system and is good for double Birkhoff averages. In this case,  $\G$ consists of all the $(m-1)$-dimensional rational subspaces of $\mathbb{R}^{m}$. By Theorem \ref{main3},  a Borel measure $\nu$ on $\mathbb{R}^{m}$ is weakly equidistributed for $\X$ if and only if the $\nu$-measure of any translation of a rational subspace of $\mathbb{R}^{m}$ is equal to zero. This recovers a special case of Theorem 1.1 of \cite{KSS}.
\end{example}	

In the case $m=1$, the assumption of the goodness for double Birkhoff averages can be dropped by using Bourgain's result \cite{B} (see Section \ref{s21}):
\begin{prop}\label{main4}
	Let $\nu$ be a Borel measure on $\mathbb{R}$ and $\X$ be an ergodic $\mathbb{R}$-system. Then $\nu$ is weakly equidistributed for $\X$ if and only if $\nu$ is atomless (meaning that $\nu(\{\beta\})=0$ for all $\beta\in\mathbb{R}$).
\end{prop}

It is an interesting question to understand the algebraic structure of $\G$.
Let $W'(\X)$ denote the collection of all $\a\in\mathbb{R}^{m}$ such that $I((n\a)_{n\in\mathbb{Z}})\neq I(\mathbb{R}^{m})$. Then $W(\X)\subseteq W'(\X)$. By a result of Pugh and Shub \cite{PS} (see also Theorem \ref{PS}), $W'(\X)$ is contained in the union of at most countably many hyperplanes of $\mathbb{R}^{m}$. We show in Section \ref{s3} an analog of this result for $\G$:
\begin{thm}\label{geo}
	Let $m\in\mathbb{N}$ and $\X$ be an $\mathbb{R}^{m}$-system. 
	Then $W(\X)$ is the union of at most countably many proper subspaces of $\mathbb{R}^{m}$.
\end{thm}
In other words, $W(\X)$ is contained in the union of at most countably many hyperplanes of $\mathbb{R}^{m}$ passing through the origin.

While all the previous mentioned results on the strong equidistribution rely heavily on tools from harmonic analysis, in this paper, we provide purely ergodic theoretic proofs for Theorems \ref{main} and \ref{main3} and Proposition \ref{main4}.
An advantage of considering the weak  equidistribution problem is that while the conditions in Theorems \ref{main}, \ref{main3} and Proposition \ref{main4} are almost necessary 
and sufficient, the conditions imposed in all the previously mentioned  results for strong equidistribution seem to be far from being necessary. Moreover, we make no smoothness assumption for the Borel measure $\nu$ in the main results of this paper, as we do not apply Fourier analysis in the proofs.

\subsection{Organization of the paper}
In Section \ref{s3}, we provide two variations of the result of Pugh and Shub \cite{PS} on the ergodic directions of $\mathbb{R}^{m}$-systems for later uses. In Section \ref{s4}, we introduce Host-Kra characteristic factors, which is the main tool of this paper. For the convenience of our purpose and future researches, we develop the existing results on this topic into a more general setting. The proves of the main results (Theorems \ref{main} and \ref{main3}) are in Section \ref{s5}. In Section \ref{s21}, we take a review for systems which are good for the double Birkhoff averages, and discuss applications of the main theorems of this paper to such systems (including the proof of Proposition \ref{main4}).

\section{Ergodic elements in Ergodic systems}\label{s3}
Let $m\in\mathbb{N}$, $\X$ be an $\mathbb{R}^{m}$-system and $H$ be a subgroup of $G$. %Let $I(H)$ denote the set of all the $H$-invariant measurable set of $\X$. 
We say that $(T_{h})_{h\in H}$ is \emph{ergodic} for $\X$ if all the $H$-invariant subsets of $\X$ are of measure either 0 or 1.

A key ingredient connecting $\mathbb{R}^{m}$-systems and $\mathbb{Z}^{m}$-systems is the following:
\begin{thm}[Pugh and Shub \cite{PS}, Theorem 1.1]\label{PS}
	Let $m\in\mathbb{N}$ and $\X$ be an ergodic $\mathbb{R}^{m}$-system. Then for all $\a\in\mathbb{R}^{m}$ except at most a countable family of hyperplanes of $\mathbb{R}^{m}$, the $\mathbb{Z}$-action $(T_{n\a})_{n\in\mathbb{Z}}$ is ergodic for $\X$.
	%Let $m,k\in\mathbb{N}$ with $0<k\leq m$. Let $\X$ be a $\mathbb{R}^{m}$-system and $H$ be a $k$-dimensional subspace of $\mathbb{R}^{m}$. Then for all but countably many choices of $\mathbb{R}$-basis $(\alpha_{1},\dots,\alpha_{k})$ of $H$, letting $H^{\ast}$ denote the $\mathbb{Z}$-span of  $(\alpha_{1},\dots,\alpha_{k})$, we have that $I(H)=I(H^{\ast})$.
\end{thm}
In this section, we provide two generalization of Pugh and Shub's Theorem.
The first is a relative version of Theorem \ref{PS}.  
\begin{lem}\label{PS2}
	Let $m\in\mathbb{N}$ and $\X$ be a (not necessarily ergodic) $\mathbb{R}^{m}$-system. Then for all $\a\in\mathbb{R}^{m}$ except at most a countable family of hyperplanes of $\mathbb{R}^{m}$, we have that $I((n\a)_{n\in\mathbb{Z}})=I(\mathbb{R}^{m})$.
\end{lem}
The proof of this lemma is almost identical to that of Theorem \ref{PS}, and so we only provide a sketch.
\begin{proof}[Sketch of the proof]
	Let $G$ be an abelian, Hausdorff, locally compact and separable group, and $\X=(X,\mathcal{X},\mu,(T_{g})_{g\in G})$ be an $G$-system. 
	Using Zorn's Lemma, and the fact that $\X$ is separable, we may decompose $L^{2}(\mu)$ as a countable direct sum of orthogonal closed subspaces
	$$L^{2}(\mu)=H\oplus_{i} H_{i},$$
	where $H$ consists of all the $G$-invariant functions, and for each $i$, $H_{i}$ is the smallest closed subspace of  $L^{2}(\mu)$ containing the $G$-orbit of some $f_{i}\in L^{2}(\mu)$. To each $i$, there corresponds a unique normalized Borel measure $\beta_{i}$ on the dual group $\widehat{G}=Hom(G,\mathbb{T}^{1})$ such that $(T_{g})_{g\in G}$ restricted to $H_{i}$ is unitarily equivalent to the ``direct integral" representation $m_{i}\colon G\to Un(L^{2}(\widehat{G},\beta_{i}))$, $$g\to\langle\cdot,g\rangle f(\cdot), f\in L^{2}(\widehat{G},\beta_{i}).$$
	For $g\in G$, denote $$\ker(g)=\{\chi\in\widehat{G}\colon\langle\chi, g\rangle=1\}.$$
	Following the proof  in \cite{PS}, we can deduce the following:
	
	\textbf{Claim 1:}
	The identity element of $\widehat{G}$ has zero $\beta_{i}$ measure for all $i$.
	
	\textbf{Claim 2:}
	If $I((g^{n})_{n\in\mathbb{Z}})\neq I(G)$ for some $g\in G$, then there exists $i$ such that $\beta_{i}(\ker (g))>0$.
	
	For Claim 1, if the identity element of $\widehat{G}$ has positive $\beta_{i}$ measure for some $i$, by the argument of the proof of Lemma 1 of \cite{PS}, one can construct a non-trivial $G$-invariant function lying in $H_{i}$, a contradiction. For Claim 2,
	if $I((g^{n})_{n\in\mathbb{Z}})\neq I(G)$ for some $g\in G$, then there exists a $g$-invariant function which does not belong to $H$, and the rest of the proof is identical to Lemma 2 of \cite{PS}.
	
	We now return to the case when $G=\mathbb{R}^{m}$. By using Claims 1 and 2 to replace Lemmas 1 and 2 of \cite{PS}, and following the same argument as in Section 5 of \cite{PS}, we finish the proof.
\end{proof}

We now prove Theorem \ref{geo}, which is a variation of Theorem \ref{PS} for the ergodicity of $\mathbb{R}$-actions. This result is interesting on its own. 

\begin{proof}[Proof of Theorem \ref{geo}]
	We first claim that for every subspace $V$ of $\mathbb{R}^{m}$, either $V\subseteq W(\X)$ or there exists a  family of at most countably many proper subspaces $(V_{j})_{j\in J}$ of $V$ such that $W(\X)\cap V\subseteq \bigcup_{j\in J}V_{j}$.
	
	Let $\X=(X,\mathcal{X},\mu,(T_{\a})_{\a\in\mathbb{R}^{m}})$ and
	suppose that $V\not\subseteq W(\X)$. Then there exists $\a\in V\backslash\G$ such that $I((t\a)_{t\in\mathbb{R}})=I(\mathbb{R}^{m})\subseteq I(V)\subseteq I((t\a)_{t\in\mathbb{R}})$. Therefore $I(\mathbb{R}^{m})=I(V)$. 	Now consider the $V$-system $\Y=(X,\mathcal{X},\mu,(T_{\a})_{\a\in V})$. Since $I(\mathbb{R}^{m})=I(V)$, we have that $W(\X)\cap V=W(\Y)$.
	
	Suppose that for every family of at most countably many proper subspaces $(V_{j})_{j\in J}$ of $V$, we have that  $W(\Y)=W(\X)\cap V\not\subseteq \bigcup_{j\in J}V_{j}$. Since $I((t\alpha)_{t\in\mathbb{R}})\subseteq I((n\alpha)_{n\in\mathbb{Z}})$, applying Lemma \ref{PS2} to $\Y$, there exist at most countably many proper subspaces $(V_{j})_{j\in J}$ of $V$, and at most countably many hyperplanes $(V_{j})_{j\in J'}$ of $V$ not passing through the origin such that
	$W(\Y)\subseteq \bigcup_{j\in J\cup J'}V_{j}$. By assumption, $W(\Y)\not\nsubseteq \bigcup_{j\in J}V_{j}$. So there exists $\alpha\in W(\Y)\backslash\bigcup_{j\in J}V_{j}$.
	By the definition of $W(\Y)$, it is easy to see that $\alpha\in W(\Y)$ implies that $t\alpha\in W(\Y)$ for all $t\in\mathbb{R}$. Since $\alpha\notin \bigcup_{j\in J}V_{j}$ implies that $t\alpha\notin \bigcup_{j\in J}V_{j}$ for all $t\neq 0$, we must have that $\{t\a\colon t\in\mathbb{R}\}\subseteq \bigcup_{j\in J'}V_{j}$. However, since $V_{j}$ does not pass through the origin for all $j\in J'$, $\{t\a\colon t\in\mathbb{R}\}\cap\bigcup_{j\in J'}V_{j}$ is a countable set, which leads to a contradiction. This proves the claim.
	
	Now we return to the proof of the theorem. By Lemma \ref{PS2}, $W(\X)\neq \mathbb{R}^{m}$. By the claim, there exists a  family of at most countably many subspaces $(V_{j})_{j\in J_{1}\cup L_{1}}$ of $\mathbb{R}^{m}$ of co-dimension at least 1 such that $W(\X)\subseteq \bigcup_{j\in J_{1}\cup L_{1}}V_{j}$, where $V_{j}\subseteq W(\X)$ if $j\in L_{1}$ and $V_{j}\not\subseteq W(\X)$ if $j\in J_{1}$. Applying the claim to each subspace in $J_{1}$, there exists a family of at most countably many subspaces $(V_{j})_{j\in J_{2}\cup L_{2}}$ of $\mathbb{R}^{m}$ such that $W(\X)\subseteq \bigcup_{j\in J_{2}\cup L_{2}}V_{j}$, where all $V_{j},j\in J$ are of co-dimension at least 2, $V_{j}\subseteq W(\X)$ if $j\in L_{2}$, and $V_{j}\not\subseteq W(\X)$ if $j\in J_{2}$. Using the claim repeatedly, there exists a family of at most countably many subspaces $(V_{j})_{j\in J_{m}\cup L_{m}}$ of $\mathbb{R}^{m}$ such that $W(\X)\subseteq \bigcup_{j\in J_{m}\cup L_{m}}V_{j}$, where all $V_{j},j\in J$ are of co-dimension at least $m$, $V_{j}\subseteq W(\X)$ if $j\in L_{m}$, and $V_{j}\not\subseteq W(\X)$ if $j\in J_{m}$. Since $J_{m}$ is an empty set, we have that $W(\X)=\bigcup_{j\in J_{m}}V_{j}$, which finishes the proof.
\end{proof}

\section{Characteristic factors and structure theorem}\label{s4}
\subsection{Host-Kra characteristic factors}
%Although we are only interested in $\mathbb{R}^{m}$-systems in this paper, many results in this section can be generalized to $G$-systems for other groups $G$. 
%In this section, $G$ is always assumed to be a locally compact Hausdorff topological group unless otherwise stated.

Let $G$ be an abelian locally compact Hausdorff topological group and $H_{1},\dots,H_{d}$ be subgroups of $G$.
Let $\X=(X,\mathcal{X},\mu,(T_{g})_{g\in G})$ be a $G$-system. 
For convenience we denote $X^{[d]}=X^{2^{d}}$, $\mathcal{X}^{[d]}=\mathcal{X}^{2^{d}}$ and  $T_{g}^{[d]}=T_{g}^{2^{d}}$. 
For any subgroup $H$ of $G$, let $I(H)$ denote the $\sigma$-algebra of $\mathcal{X}$ consisting of all the $H$-invariant sets. For $1\leq j\leq d-1$, let $I_{\Delta}(H^{[j]}_{j+1})$ denote the sub $\sigma$-algebra of $\mathcal{X}^{[j]}$ consisting of all the sets which are invariant under $T^{[j]}_{g}$ for all $g\in H_{j+1}$. 
We inductively define the \emph{Host-Kra measures} $\mu_{H_{1},\dots,H_{j}}$ on $X^{[j]}$ by setting
$\mu_{H_{1}}=\mu\times_{I(H_{1})}\mu,$ 
meaning that 
$$\int_{X^{2}}f\otimes g\,d\mu_{H_{1}}=\int_{X}\mathbb{E}(f\vert I(H_{1}))\cdot \mathbb{E}(g\vert I(H_{1}))\,d\mu$$ for all $f,g\in L^{\infty}(\mu),$
and for all $1\leq j\leq d-1$, define
$\mu_{H_{1},\dots,H_{j+1}}=\mu_{H_{1},\dots,H_{j}}\times_{I_{\Delta}(H^{[j]}_{j+1})}\mu_{H_{1},\dots,H_{j}},$
meaning that 
$$\int_{X^{[j+1]}}F\otimes G\,d\mu_{H_{1},\dots,H_{j+1}}=\int_{X^{[j]}}\mathbb{E}(F\vert I_{\Delta}(H^{[j]}_{j+1}))\cdot \mathbb{E}(G\vert I_{\Delta}(H^{[j]}_{j+1}))\,d\mu_{H_{1},\dots,H_{j}}$$ for all $F,G\in L^{\infty}(\mu^{[j]}).$
We define the \emph{Host-Kra seminorm} by 
$$\Vert f\Vert_{\X,H_{1},\dots,H_{d}}:=\Bigl(\int_{X^{[d]}}f^{\otimes 2^{d}}\,d\mu_{H_{1},\dots,H_{d}}\Bigr)^{\frac{1}{2^{d}}}$$
for all $f\in L^{\infty}(\mu)$.
Let $Z_{H_{1},\dots,H_{d}}(\X)$ (or $Z_{H_{1},\dots,H_{d}}$ when there is no confusion) be the sub $\sigma$-algebra of $\mathcal{X}$ such that for all $f\in L^{\infty}(\mu)$, $$\mathbb{E}(f\vert Z_{H_{1},\dots,H_{d}}(\X))=0 \text{ if and only if } \Vert f\Vert_{\X,H_{1},\dots,H_{d}}=0.$$
Similar to the proof of Lemma 4 of \cite{H} (or Lemma 4.3 of \cite{HK}), one can show that  $Z_{H_{1},\dots,H_{d}}$ is well defined and we call it a \emph{Host-Kra characteristic factor}.\footnote{Sometimes we will slight abuse the notation and say that "$Z_{H_{1},\dots,H_{d}}$ is a factor $X$", meaning that the system  $(X,Z_{H_{1},\dots,H_{d}},\mu,G)$ is a factor of $(X,\mathcal{X},\mu,G)$.}

The following lemma is useful in many circumstances:
\begin{lem}\label{replacement0}
	Let $G$ be an abelian locally compact Hausdorff topological group and $\X$ be a $G$-system. Let $H_{1},\dots,H_{d},H'_{j}$ be subgroups of $G$ for some $1\leq j\leq d$. 
	
	(i) For every permutation $\sigma\colon\{1,\dots,d\}\to\{1,\dots,d\}$, we have that $Z_{H_{1},\dots,H_{d}}(\X)=Z_{H_{\sigma(1)},\dots,H_{\sigma(d)}}(\X)$;
	
	(ii) If $I(H_{j})=I(H'_{j})$, then $Z_{H_{1},H_{2},\dots,H_{j},\dots,H_{d}}(\X)=Z_{H_{1},H_{2},\dots,H'_{j},\dots,H_{d}}(\X)$.
\end{lem}
\begin{proof}
	(i) The proof is similar to \cite{H} and so we only provide a sketch. It suffices to show that for all subgroups $H_{1},\dots, H_{d}$ of $G$ and $1\leq i\leq d-1$, we have that $Z_{H_{1},\dots,H_{i},H_{i+1},\dots,H_{d}}(\X)=Z_{H_{1},\dots,H_{i+1},H_{i},\dots,H_{d}}(\X)$, or $$\Vert f\Vert_{\X,H_{1},\dots,H_{i},H_{i+1},\dots,H_{d}}=0\Leftrightarrow\Vert f\Vert_{\X,H_{1},\dots,H_{i+1},H_{i},\dots,H_{d}}=0$$ for all $f\in L^{\infty}(\mu)$.\footnote{It seems that the proof of Proposition 3 of \cite{H} can be adapted to proving that $f\Vert_{\X,H_{1},\dots,H_{i},H_{i+1},\dots,H_{d}}=\Vert f\Vert_{\X,H_{1},\dots,H_{i+1},H_{i},\dots,H_{d}}$ for all $f\in L^{\infty}(X)$. But we do not need this property in this paper.} By the definition of the Host-Kra measure, it suffices to show that 
	$\Vert f\Vert_{\X,H_{1},\dots,H_{i},H_{i+1}}=0\Leftrightarrow\Vert f\Vert_{\X,H_{1},\dots,H_{i+1},H_{i}}=0$. Replacing the system $\X=(X,\mathcal{X},\mu,(T_{g})_{g\in G})$ with $(X^{[i-1]},\mathcal{X}^{[i-1]},\mu_{H_{1},\dots,H_{i-1}},(T^{[i-1]}_{g})_{g\in G})$, it suffices to show that for all $G$-system $\X$ and subgroups $H_{1}, H_{2}$ of $G$, we have that $$\Vert f\Vert_{\X,H_{1},H_{2}}=0\Leftrightarrow\Vert f\Vert_{\X,H_{2},H_{1}}=0.$$ 
	
	Suppose first that $\Vert f\Vert_{\X,H_{1},H_{2}}=0$. We may assume that $\Vert f\Vert_{L^{\infty}(\mu)}\leq 1$.
	Let $(F_{1,n})_{n\in\mathbb{N}}$ and $(F_{2,n})_{n\in\mathbb{N}}$ be any F\o lner sequences of $H_{1}$ and $H_{2}$, respectively. Similar to Lemma 2 of \cite{H}, it is not hard to show that 
	$$\Bigl\vert\lim_{N\to\infty}\frac{1}{\vert F_{1,N}\vert\cdot \vert F_{2,N}\vert}\sum_{g_{1}\in F_{1,N},g_{2}\in F_{2,N}}\int_{X}f\cdot T_{g_{1}}f\cdot T_{g_{2}}f\cdot T_{g_{1}g_{2}}f\,d\mu\Bigr\vert\leq \Vert f\Vert_{\X,H_{1},H_{2}}=0.$$
	So the limit $\lim_{N\to\infty}\frac{1}{\vert F_{1,N}\vert\cdot \vert F_{2,N}\vert}\sum_{g_{1}\in F_{1,N},g_{2}\in F_{2,N}}\int_{X}f\cdot T_{g_{1}}f\cdot T_{g_{2}}f\cdot T_{g_{1}g_{2}}f\,d\mu$ exists and equals to 0. On the other hand, similar to (11) of \cite{H} (and invoke Theorem 8.13 of \cite{ET}, the Birkhoff ergodic theorem for $G$-systems), 
	$$\Vert f\Vert_{\X,H_{2},H_{1}}^{4}=\lim_{N\to\infty}\frac{1}{\vert F_{2,N}\vert}\sum_{g_{2}\in F_{2,N}}\lim_{N\to\infty}\frac{1}{\vert F_{1,N}\vert}\sum_{g_{1}\in F_{1,N}}\int_{X}f\cdot T_{g_{1}}f\cdot T_{g_{2}}f\cdot T_{g_{1}g_{2}}f\,d\mu,$$
	where the limit 
	$$\lim_{N\to\infty}\frac{1}{\vert F_{1,N}\vert}\sum_{g_{1}\in F_{1,N}}\int_{X}f\cdot T_{g_{1}}f\cdot T_{g_{2}}f\cdot T_{g_{1}g_{2}}f\,d\mu$$
	exists for all $g_{2}\in H_{2}$. By Lemma 1.1 and 1.2 of \cite{BL}, 
	\begin{equation}\nonumber
		\begin{split}
			&\quad \lim_{N\to\infty}\frac{1}{\vert F_{2,N}\vert}\sum_{g_{2}\in F_{2,N}}\lim_{N\to\infty}\frac{1}{\vert F_{1,N}\vert}\sum_{g_{1}\in F_{1,N}}\int_{X}f\cdot T_{g_{1}}f\cdot T_{g_{2}}f\cdot T_{g_{1}g_{2}}f\,d\mu
			\\&=\lim_{N\to\infty}\frac{1}{\vert F_{1,N}\vert\cdot \vert F_{2,N}\vert}\sum_{g_{1}\in F_{1,N},g_{2}\in F_{2,N}}\int_{X}f\cdot T_{g_{1}}f\cdot T_{g_{2}}f\cdot T_{g_{1}g_{2}}f\,d\mu=0,
		\end{split}
	\end{equation}
	and so $\Vert f\Vert_{\X,H_{2},H_{1}}=0$. Similarly,  $\Vert f\Vert_{\X,H_{2},H_{1}}=0$ implies that $\Vert f\Vert_{\X,H_{1},H_{2}}=0$.
	
	%	We may further assume that the action $(T_{g})_{g\in H_{1}+H_{2}}$ is ergodic on $\X$. Indeed,
	%	let %$\mathcal{W}$ be the sub $\sigma$-algebra of $\mathcal{X}$ consisting of all the %$(H_{1}+H_{2})$-invariant sets and let 
	%	$$\mu=\int\mu^{\omega}\,dP(\omega)$$
	%	be the ergodic decomposition of $\mu$ under the action $H_{1}+H_{2}$. Since $I(H_{1}+H_{2})\subseteq I(H_{1})$, we have that
	%	$$\mu_{H_{1}}=\mu\times_{I(H_{1})}\mu=\int \mu^{\omega}\times_{I(H_{1})}\mu^{\omega}\,dP(\omega)=\int \mu^{\omega}_{ H_{1}}\,dP(\omega).$$
	%	Since $I(H_{1}+H_{2})\otimes I(H_{1}+H_{2})\subseteq I_{\Delta}(H^{[1]}_{2})$, we have that 
	%	$$\mu_{H_{1},H_{2}}=\mu_{H_{1}}\times_{I_{\Delta}(H^{[1]}_{2})}\mu_{H_{1}}=\int \mu^{\omega}_{H_{1}}\times_{I_{\Delta}(H^{[1]}_{2})}\mu^{\omega}_{H_{1}}\,dP(\omega)=\int \mu^{\omega}_{H_{1},H_{2}}\,dP(\omega).$$
	%	Similarly, 
	%	$$\mu_{H_{2},H_{1}}=\int \mu^{\omega}_{H_{2},H_{1}}\,dP(\omega).$$
	%	If $\nu_{H_{1},H_{2}}=\nu_{H_{2},H_{1}}$ for all ergodic system $(X,\mathcal{X},\nu,(T_{g})_{g\in H_{1}+H_{2}}$, then $\mu^{\omega}_{H_{1},H_{2}}=\mu^{\omega}_{H_{2},H_{1}}$ for all $\omega$ and so $\mu_{H_{1},H_{2}}=\mu_{H_{2},H_{1}}$. Therefore, we may assume that $H_{1}+H_{2}=G$ and $\X$ is an ergodic $G$-system.

	We now prove (ii). By (i), we may assume without loss of generality that $j=1$. Note that $$\mu_{H_{1}}=\mu\times_{I(H_{1})}\mu=\mu\times_{I(H'_{1})}\mu=\mu_{H'_{1}}.$$ By induction, $\mu_{H_{1},H_{2},\dots,H_{d}}=\mu_{H'_{1},H_{2},\dots,H_{d}}$ and so $Z_{H_{1},H_{2},\dots,H_{d}}(\X)=Z_{H'_{1},H_{2},\dots,H_{d}}(\X)$, which finishes the proof.	 
\end{proof}
The following is an immediate corollary of Lemma \ref{replacement0}:
\begin{lem}\label{replacement}
	Let $G$ be an abelian locally compact Hausdorff topological group and $\X$ be a $G$-system. Let $H_{1},\dots,H_{d}$, $H'_{1},\dots,H'_{d}$ be subgroups of $G$. If $I(H_{i})=I(H'_{i})$ for all $1\leq i\leq d$, then $Z_{H_{1},H_{2},\dots,H_{d}}(\X)=Z_{H'_{1},H'_{2},\dots,H'_{d}}(\X)$. 
\end{lem}

\subsection{Structure theorems for $\mathbb{R}^{m}$-systems}
In this section, we establish structure theorems for $\mathbb{R}^{m}$-systems. These questions has been studied in various papers, see for example \cite{A,BLM,P,Z}. As none of the existing results can be applied directly to our problem, we need to develop the past results into a more general setting. In this paper, we only use some special cases of the theorems developed in this section. But we still write all the results in full generality for the purpose of future researches.

\begin{conv}
	Let $m\in\mathbb{N}$,  $\X$ be an $\mathbb{R}^{m}$-system and $H_{1},\dots,H_{d}$ be subgroups of $\mathbb{R}^{m}$. In the notations $\mu_{H_{1},H_{2},\dots,H_{d}}$, $Z_{H_{1},H_{2},\dots,H_{d}}$ and $\Vert\cdot\Vert_{\X,H_{1},H_{2},\dots,H_{d}}$, if $H_{i}=(t\a_{i})_{t\in\mathbb{R}}$ for some $\a_{i}\in\mathbb{R}^{m}$, we abbreviate $H_{i}$ by $\a_{i}$. If  $H_{i}=(n\a_{i})_{n\in\mathbb{Z}}$ for some $\a_{i}\in\mathbb{R}^{m}$, we abbreviate $H_{i}$ by $\ha_{i}$. For example, the notion $Z_{\a_{1},\a_{2},\ha_{3}}$ represents $Z_{(t\a_{1})_{t\in\mathbb{R}},(t\a_{2})_{t\in\mathbb{R}},(n\a_{3})_{n\in\mathbb{Z}}}$, and $\mu_{\a_{1},\a_{2},\ha_{3}}$ represents $\mu_{(t\a_{1})_{t\in\mathbb{R}},(t\a_{2})_{t\in\mathbb{R}},(n\a_{3})_{n\in\mathbb{Z}}}$.
\end{conv}

The Host-Kra characteristic factor is an important tool in the study of problems related to multiple averages. For example, certain Host-Kra characteristic factors control the $L^{2}$ limit of  multiple averages for $\mathbb{Z}^{m}$-systems:

\begin{thm}\label{H}
	Let $m\in\mathbb{N}$, $\X=(X,\mathcal{X},\mu,(T_{g})_{g\in\mathbb{R}^{m}})$ be an $\mathbb{R}^{m}$-system and let $\alpha_{1},\dots,\alpha_{d}\in\mathbb{R}^{m}$. Denote $\widehat{Z}_{i}:=Z_{\ha_{i},\widehat{\alpha_{1}-\alpha_{i}},\dots,\widehat{\alpha_{d}-\alpha_{i}}}(\X)$ for all $1\leq i\leq d$. Then for all $f_{1},\dots,f_{d}\in L^{\infty}(\mu)$, both the  $L^{2}(\mu)$ limits of
	$$\lim_{N\to\infty}\frac{1}{N}\sum_{n=0}^{N-1}f_{1}(T_{n\alpha_{1}}x)\cdot\ldots\cdot f_{d}(T_{n\alpha_{d}x})$$
	 and 
	 $$\lim_{N\to\infty}\frac{1}{N}\sum_{n=0}^{N-1}\mathbb{E}(f_{1}\vert \widehat{Z}_{1})(T_{n\alpha_{1}}x)\cdot\ldots\cdot \mathbb{E}(f_{d}\vert \widehat{Z}_{d})(T_{n\alpha_{d}}x)$$
	 exist and coincide (as $L^{2}(\mu)$ functions). Moreover, if both limit exist for $\mu$-a.e. $x\in X$, then they coincide for  $\mu$-a.e. $x\in X$. 
\end{thm} 
\begin{proof}
	The existence and coincidence of the $L^{2}(\mu)$ limits is a result of Host (\cite{H}, Proposition 1). The existence and coincidence of the pointwise limit follows from the fact that if a sequence of bounded functions converge both as $L^{2}(\mu)$ functions and almost everywhere, then both limits are the same.
\end{proof}

The following lemma illustrates the connection between  Host-Kra measures, seminorms and characteristic factors for $\mathbb{R}^{m}$-systems and that for $\mathbb{Z}^{m}$-systems.
\begin{lem}\label{picks}
	Let $\X$ be an $\mathbb{R}^{m}$-system and $\a_{1},\dots,\a_{d}\in\mathbb{R}^{m}$. Then for Lebesgue almost every $s\in\mathbb{R}$, $\mu_{\widehat{s\a_{1}},\dots,\widehat{s\a_{d}}}=\mu_{\a_{1},\dots,\a_{d}}$, $\Vert \cdot\Vert_{\X,\widehat{s\a_{1}},\dots,\widehat{s\a_{d}}}=\Vert\cdot\Vert_{\X,\a_{1},\dots,\a_{d}}$ and 
	$Z_{\widehat{s\a_{1}},\dots,\widehat{s\a_{d}}}=Z_{\a_{1},\dots,\a_{d}}$.
\end{lem}
\begin{proof}
	Let $H_{i}$ denote the $\mathbb{R}$-span of $\a_{i}$ and $\widehat{H}_{i}$ denote the $\mathbb{Z}$-span of $\ha_{i}$. Applying Lemma \ref{PS2} to each $H_{i}$, we have that for Lebesgue almost every $s\in\mathbb{R}$, we have that $I(H_{i})=I(s\widehat{H}_{i})$ for all $1\leq i\leq d$. By Corollary \ref{replacement},  $Z_{s\widehat{H}_{1},\dots,s\widehat{H}_{d}}=Z_{H_{1},\dots,H_{d}}$. By definition, 
	$\mu_{s\widehat{H}_{1},\dots,s\widehat{H}_{d}}=\mu_{H_{1},\dots,H_{d}}$
	and 
	$\Vert \cdot\Vert_{\X,s\widehat{H}_{1},\dots,s\widehat{H}_{d}}=\Vert\cdot\Vert_{\X,H_{1},\dots,H_{d}}$. This finishes the proof.
\end{proof}

We can now prove
the following analogue of Theorem \ref{H} for $\mathbb{R}^{m}$-systems:
\begin{prop}\label{Hf}
	Let $m\in\mathbb{N}$, $\X=(X,\mathcal{X},\mu,(T_{g})_{g\in\mathbb{R}^{m}})$ be an $\mathbb{R}^{m}$-system and let $\alpha_{1},\dots,\alpha_{d}\in\mathbb{R}^{m}$. Denote $Z_{i}:=Z_{\alpha_{i},\alpha_{1}-\alpha_{i},\dots,\alpha_{d}-\alpha_{i}}(\X)$ for all $1\leq i\leq d$. Then for all $f_{1},\dots,f_{d}\in L^{\infty}(\mu)$, both the limits 
	$$\lim_{T\to\infty}\frac{1}{T}\int_{0}^{T}f_{1}(T_{t\alpha_{1}}x)\cdot\ldots\cdot f_{d}(T_{t\alpha_{d}}x)\,dt$$
	and 
	$$\lim_{T\to\infty}\frac{1}{T}\int_{0}^{T}\mathbb{E}(f_{1}\vert Z_{1})(T_{t\alpha_{1}}x)\cdot\ldots\cdot\mathbb{E}(f_{d}\vert Z_{d})(T_{t\alpha_{d}}x)\,dt$$
exist and coincide (as $L^{2}(\mu)$ functions). Moreover, if both limit exists for $\mu$-a.e. $x\in X$, then they coincide for $\mu$-a.e. $x\in X$.
\end{prop} 
\begin{proof}
By Lemma \ref{picks}, there exists $s\in\mathbb{R}$ such that $Z_{i}=Z_{\alpha_{i},\alpha_{1}-\alpha_{i},\dots,\alpha_{d}-\alpha_{i}}(\X)=Z_{\wh{s\alpha_{i}},\wh{s(\alpha_{1}-\alpha_{i})},\dots,\wh{s(\alpha_{d}-\alpha_{i})}}(\X)$	for all $1\leq i\leq d$. For convenience we may assume without loss of generality that $s=1$. By Theorem \ref{H}, for all $f_{1},\dots,f_{d}\in L^{\infty}(\mu)$, 
\begin{equation}\nonumber
\begin{split}
\lim_{N\to\infty}\frac{1}{N}\sum_{n=0}^{N-1}T_{n\alpha_{1}}f_{1}\cdot\ldots\cdot T_{n\alpha_{d}}f_{d}=\lim_{N\to\infty}\frac{1}{N}\sum_{n=0}^{N-1}T_{n\alpha_{1}}\mathbb{E}(f_{1}\vert Z_{1})\cdot\ldots\cdot T_{n\alpha_{d}}\mathbb{E}(f_{d}\vert Z_{d}),
\end{split}
\end{equation}
where the limits are taken in $L^{2}(\mu)$. Using the fact that every $Z_{i}$ is $G$-invariant, we have that as $L^{2}(\mu)$ functions,
\begin{equation}\label{1}
\begin{split}
&\quad\lim_{T\to\infty}\frac{1}{T}\int_{0}^{T}f_{1}(T_{t\alpha_{1}}x)\cdot\ldots\cdot f_{d}((T_{t\alpha_{d}}x))\,dt
\\&=\int_{0}^{1}\lim_{N\to\infty}\frac{1}{N}\sum_{n=0}^{N-1}(T_{r\alpha_{1}}f_{1})(T_{n\alpha_{1}}x)\cdot\ldots\cdot (T_{r\alpha_{d}}f_{d})(T_{n\alpha_{d}}x)\,dr
\\&=\int_{0}^{1}\lim_{N\to\infty}\frac{1}{N}\sum_{n=0}^{N-1}\mathbb{E}(T_{r\alpha_{1}}f_{1}\vert Z_{1})(T_{n\alpha_{1}}x)\cdot\ldots\cdot \mathbb{E}(T_{r\alpha_{d}}f_{d}\vert Z_{d})(T_{n\alpha_{d}}x)\,dr
\\&=\lim_{T\to\infty}\frac{1}{T}\int_{0}^{T}\mathbb{E}(f_{1}\vert Z_{1})(T_{t\alpha_{1}}x)\cdot\ldots\cdot \mathbb{E}(f_{d}\vert Z_{d})(T_{t\alpha_{d}}x)\,dt.
\end{split}
\end{equation}
Note that if a sequence of bounded functions converge both as $L^{2}(\mu)$ functions and almost everywhere, then both limits are the same. So (\ref{1}) also holds for $\mu$-a.e. $x\in X$ if all the limits in (\ref{1}) exist for $\mu$-a.e. $x\in X$.
\end{proof}
%Using the same proof of Proposition \ref{Hf}, we get the following $\mathbb{R}^{m}$ analogue of Theorem \ref{DS}:
%	\begin{prop}\label{thm:DSf}
%		Let $\X$ be a distal $\mathbb{R}^{m}$-system and let $\alpha_{1},\dots,\alpha_{d}\in\mathbb{R}^{m}$. Denote $Z_{i}:=Z_{\alpha_{i},\alpha_{1}-\alpha_{i},\dots,\alpha_{d}-\alpha_{i}}(\X)$ for all $1\leq i\leq d$. Then for all $f_{1},\dots,f_{d}\in L^{\infty}(\mu)$, 
%		\begin{equation}\nonumber
%		\begin{split}
%		\lim_{T\to\infty}\frac{1}{T}\int_{0}^{T}T_{t\alpha_{1}}f_{1}\cdot\ldots\cdot T_{t\alpha_{d}}f_{d}\,dt=\lim_{T\to\infty}\frac{1}{T}\int_{0}^{T}T_{t\alpha_{1}}\mathbb{E}(f_{1}\vert Z_{1})\cdot\ldots\cdot T_{t\alpha_{d}}\mathbb{E}(f_{d}\vert Z_{d})\,dt
%		\end{split}
%		\end{equation}
%		almost everywhere  and both limits exist.
%	\end{prop}

Let $X=N/\Gamma$, where $N$ is a ($k$-step) nilpotent group and $\Gamma$ is a discrete cocompact subgroup of $N$. Let $\mathcal{X}$ and $\mu$ be the Borel $\sigma$-algebra and Haar measure of $X$. Let $T_{g}\colon X\to X$, $T_{g}x=b_{g}\cdot x, g\in G$ for some group homomorphism $g\to b_{g}$ from $G$ to $N$. We say that $\X=(X,\mathcal{X},\mu,(T_{g})_{g\in G})$ is a \emph{($k$-step) $G$-nilsystem}. It is classical that we can choose $N$ to be simply connected, and we make this assumption throughout this paper. We remark that if $G$ is connected, then we may also assume that $N$ is connected.
The following theorem is a combination of Theorem \ref{H} in this paper and Theorem 3.7 of \cite{Z}. We omit the proof:
\begin{thm}\label{Z0}
	Let $m\in\mathbb{N}$, $\X$ be an ergodic $\mathbb{R}^{m}$-system and let $\alpha_{1},\dots,\alpha_{d}\in\mathbb{R}^{m}$. If the $\mathbb{Z}$-action $(T_{n\a_{i}})_{n\in\mathbb{Z}}$ is ergodic for $\X$ for all $1\leq i\leq d$, then $Z_{\widehat{\a_{1}},\dots,\widehat{\a_{d}}}(\X)$ is an inverse limit of $(d-1)$-step $\mathbb{R}^{m}$-nilsystems.
\end{thm}

%\begin{thm}\label{Z0}
%	Let $m\in\mathbb{N}$, $\X$ be an ergodic $\mathbb{R}^{m}$-system and let $\alpha_{1},\dots,\alpha_{d}\in\mathbb{R}^{m}$. Denote $\wh{Z}_{i}:=Z_{\ha_{i},\wh{\alpha_{1}-\alpha_{i}},\dots,\wh{\alpha_{d}-\alpha_{i}}}$ for all $1\leq i\leq d$. If the $\mathbb{Z}$-action $(T_{n\beta})_{n\in\mathbb{Z}}$ is ergodic for $\X$ for all $\beta=\alpha_{i},\alpha_{i}-\alpha_{j}$, $1\leq i,j\leq d, i\neq j$, then $\wh{Z}_{1}=\dots=\wh{Z}_{d}$ is an inverse limit of $(d-1)$-step $\mathbb{R}^{m}$-nilsystems.
%\end{thm}

We have the following structure theorem for $\mathbb{R}^{m}$-actions, which should be viewed as an analogue of the Host-Kra structure theorem \cite{HK}.

\begin{prop}\label{Z}
	Let $m\in\mathbb{N}$ and $\X$ be an ergodic $\mathbb{R}^{m}$-system.
	Then $Z_{\mathbb{R}^{m},\dots,\mathbb{R}^{m}}(\X)$ with $d$-copies of $\mathbb{R}^{m}$ is an inverse limit of $(d-1)$-step $\mathbb{R}^{m}$-nilsystems.	
	 Moreover, if $\alpha_{1},\dots,\alpha_{d}\in\mathbb{R}^{m}$ are such that the $\mathbb{R}$-action $(T_{t\alpha_{i}})_{t\in\mathbb{R}}$ is ergodic for $\X$ for all $1\leq i\leq d$, then $Z_{\alpha_{1},\dots,\alpha_{d}}(\X)=Z_{\mathbb{R}^{m},\dots,\mathbb{R}^{m}}(\X)$ with $d$-copies of $\mathbb{R}^{m}$.
\end{prop}
\begin{proof}
		By Lemma \ref{PS2}, it is not hard to show that there exist a $\mathbb{Z}$-action $(T_{n\a_{i}})_{n\in\mathbb{Z}}$ ergodic for $\X$ for all $1\leq i\leq d$. By Lemma \ref{replacement}, we have that $Z_{\mathbb{R}^{m},\dots,\mathbb{R}^{m}}=Z_{\ha_{i},\wh{\alpha_{1}-\alpha_{i}},\dots,\wh{\alpha_{d}-\alpha_{i}}}$ with $d$-copies of $\mathbb{R}^{m}$, which is an inverse limit of $(d-1)$-step $\mathbb{R}^{m}$-nilsystems by Theorem \ref{Z0}.
	
%	If $\alpha_{1},\dots,\alpha_{d}\in\mathbb{R}^{m}$ are such that the $\mathbb{R}$-action $(T_{t\alpha_{i}})_{t\in\mathbb{R}}$ is ergodic for $\X$ for all $1\leq i\leq d$,
%	then by Lemma \ref{PS2}, there exists $s\in\mathbb{R}$ such that $I((n\a_{i})_{n\in\mathbb{Z}})=I((t\a_{i})_{t\in\mathbb{R}})=I(\mathbb{R}^{m})$ for all $1\leq i\leq d$.  By Lemma \ref{replacement}, we have that $Z_{\alpha_{1},\dots,\alpha_{d}}(\X)=Z_{\mathbb{R}^{m},\dots,\mathbb{R}^{m}}(\X)$.
If $\alpha_{1},\dots,\alpha_{d}\in\mathbb{R}^{m}$ are such that the $\mathbb{R}$-action $(T_{t\alpha_{i}})_{t\in\mathbb{R}}$ is ergodic for $\X$ for all $1\leq i\leq d$,
then $I((t\a_{i})_{t\in\mathbb{R}})=I(\mathbb{R}^{m})$ for all $1\leq i\leq d$.  By Lemma \ref{replacement}, we have that $Z_{\alpha_{1},\dots,\alpha_{d}}(\X)=Z_{\mathbb{R}^{m},\dots,\mathbb{R}^{m}}(\X)$.
\end{proof}

%\subsection{Concatenation theorem}
%We need the following result:
%\begin{thm}\label{ct}
%		Let $G$ be an  abelian group and $\X$ be an ergodic $G$-system. Let $d\in\mathbb{N}$ and $H_{1},\dots,H_{d},,H'_{1},\dots,H'_{d'}$ be subgroups of $G$. Then
%		$$Z_{H_{1},\dots,H_{d}}(\X)\cap Z_{H'_{1},\dots,H'_{d'}}(\X)\subseteq Z_{(H_{i}+H'_{i'})_{1\leq i\leq d, 1\leq i'\leq d'}}(\X),$$
%		where $H+H'$ is the group generated by $H$ and $H'$.
%\end{thm}
%{\color{red} See the other note of mine.}
%
%The following corollary is immediate:
%\begin{cor}\label{ct2}
%	Let $G$ be an abelian group and $\X$ be an ergodic $G$-system. Let $m,d_{1},\dots,d_{m}\in\mathbb{N}$ and $H_{i,j}, 1\leq i\leq m, 1\leq j\leq d_{i}$ be subgroups of $G$. Then
%	$$\bigcap_{i=1}^{m}Z_{H_{i,1},\dots,H_{i,d_{i}}}(\X)\subseteq Z_{(\sum_{i=1}^{m}H_{i,n_{i}})_{1\leq n_{i}\leq d_{i}}}(\X).$$
%\end{cor}

\section{Proof of the main theorems}\label{s5}
We prove Theorems \ref{main} and \ref{main3} in this section.
\begin{lem}\label{m0}
	Let $m\in\mathbb{N}$, $\nu$ be a Borel measure on $\mathbb{R}^{m}$ and $\X$ be an  $\mathbb{R}^{m}$-system. If $\nu(W(\X)+\beta)=0$ for all hyperplane $\beta\in\mathbb{R}^{m}$, then the set of all $(\alpha,\beta)\in\mathbb{R}^{2m}$ such that $Z_{\alpha,\alpha-\beta}=Z_{\beta,\alpha-\beta}=Z_{\mathbb{R}^{m},\mathbb{R}^{m}}$
	is of $\nu\times\nu$-measure one. 
\end{lem}
\begin{proof}
	 By Proposition \ref{Z}, if 
	 all the three $\mathbb{R}$-actions $(T_{t\a})_{t\in\mathbb{R}}$, $(T_{t\beta})_{t\in\mathbb{R}}$ and $(T_{t(\a-\beta)})_{t\in\mathbb{R}}$ are ergodic for $\X$, then $Z_{\alpha,\alpha-\beta}=Z_{\beta,\alpha-\beta}=Z_{\mathbb{R}^{m},\mathbb{R}^{m}}$. So it suffices to show that %for all  hyperplane $\ell_{0}$ of  $\mathbb{R}^{m}$,
		the sets
		$$E_{1}=\{(\alpha,\beta)\in\mathbb{R}^{2m}\colon \alpha\in W(\X)\}, E_{2}=\{(\alpha,\beta)\in\mathbb{R}^{2m}\colon \beta\in W(\X)\},$$ 
		and
		$$E_{3}=\{(\alpha,\beta)\in\mathbb{R}^{2m}\colon \alpha-\beta\in W(\X)\},$$  
		have zero $\nu\times\nu$ measure. Obviously, $$\nu\times\nu(E_{1})=\nu\times\nu(E_{2})=\nu(W(\X))=0.$$
		On the other hand,
		\begin{equation}\nonumber
		\begin{split}
		&\quad\nu\times\nu(E_{3})=\int_{\mathbb{R}^{m}}\nu(\{\alpha\in\mathbb{R}^{m}\colon \alpha\in W(\X)+\beta\})\,d\nu(\beta)
		\\&=\int_{\mathbb{R}^{m}}\nu(W(\X)+\beta)\,d\nu(\beta)=\int_{\mathbb{R}^{m}}0\,d\nu(\beta)=0.
		\end{split}
		\end{equation}
		This finishes proof.
\end{proof}

We start with a special case of Theorem \ref{main3}:
\begin{prop}\label{main21}
	Let $d, m\in\mathbb{N}$, $\nu$ be a Borel measure on $\mathbb{R}^{m}$ and $\X$ be an ergodic $\mathbb{R}^{m}$-system such that $Z_{\mathbb{R}^{m},\dots,\mathbb{R}^{m}}(\X)=\X$ with $d$-copies of $\mathbb{R}^{m}$. If $\nu(W(\X)+\beta)=0$ for all hyperplane $\beta\in\mathbb{R}^{m}$, then $\nu$ is weakly equidistributed for $\X$.
\end{prop}
\begin{proof}
It  suffices  to show that for all $f\in L^{\infty}(\mu)$ with $\int_{X}f\,d\mu=0$, 
	\begin{equation}\nonumber
	\begin{split}
	\lim_{T\to\infty}\frac{1}{T}\int_{0}^{T}\vert\int_{\mathbb{R}^{m}}f(T_{t\alpha}x)\,d\nu(\alpha)\vert^{2}\,dt=0
	\end{split}
	\end{equation}
	for $\mu$-a.e. $x\in X$. Let $J_{0}$ denote the set of all $(\alpha,\beta)\in\mathbb{R}^{2m}$ such that $Z_{\alpha,\alpha-\beta}=Z_{\beta,\alpha-\beta}=Z_{\mathbb{R}^{m},\mathbb{R}^{m}}$. By Lemma \ref{m0}, $\nu\times\nu(J_{0})=1$.
	So
	\begin{equation}\label{equ:24}
	\begin{split}
	&\quad \lim_{T\to\infty}\frac{1}{T}\int_{0}^{T}\vert\int_{\mathbb{R}^{m}}f(T_{t\alpha}x)\,d\nu(\alpha)\vert^{2}\,dt
	\\&=\lim_{T\to\infty}\frac{1}{T}\int_{0}^{T}\int_{\mathbb{R}^{2m}}f(T_{t\alpha}x)\overline{f}(T_{t\beta}x)\,d\nu(\alpha)d\nu(\beta)\,dt
	\\&=\int_{\mathbb{R}^{2m}}\Bigl(\lim_{T\to\infty}\frac{1}{T}\int_{0}^{T}f(T_{t\alpha}x)\overline{f}(T_{t\beta}x)\,dt\Bigr)\,d\nu(\alpha)d\nu(\beta)
		\\&=\int_{J_{0}}\Bigl(\lim_{T\to\infty}\frac{1}{T}\int_{0}^{T}f(T_{t\alpha}x)\overline{f}(T_{t\beta}x)\,dt\Bigr)\,d\nu(\alpha)d\nu(\beta).
	\end{split}
	\end{equation}
	
	Since $Z_{\mathbb{R}^{m},\dots,\mathbb{R}^{m}}(\X)=\X$,
	by Proposition \ref{Z} and an approximation argument, we may assume without loss of generality that $\X$ is an $\mathbb{R}^{m}$-nilsystem.
	We may assume without loss of generality that $\X$ is connected.  
	
	Suppose that $X=N/\Gamma$, where $N$ is a ($k$-step) nilpotent group and $\Gamma$ is a discrete cocompact subgroup of $N$. Let $\mathcal{X}$ and $\mu$ be the Borel $\sigma$-algebra and Haar measure of $X$. Assume that $T_{g}\colon X\to X$, $T_{g}x=b_{g}\cdot x, g\in G$ for some group homomorphism $g\to b_{g}$ from $G$ to $N$.
	Let $J$ denote the set of all $(\alpha,\beta)\in\mathbb{R}^{2m}$ such that $((b_{t\alpha}\Gamma,b_{t\beta}\Gamma))_{t\in\mathbb{R}}$ is equidistributed on $X\times X$. If $\nu\times\nu(J)=1$, then $$(\ref{equ:24})=\int_{J}\int_{X\times X}f\otimes \overline{f}\,d\mu\times\mu\,d\nu(\alpha)d\nu(\beta)=\int_{J}\Bigl\vert\int_{X}f\,d\mu\Bigr\vert^{2}\,d\nu(\alpha)d\nu(\beta)=0,$$
	which finishes the proof.
	
	We now prove that $\nu\times\nu(J)=1$.
	Since $\mathbb{R}^{m}$ is a connected group, we may assume that $X=N/\Gamma$ with $N$ being connected and simply connected. 
	Note that for all nontrivial horizontal character $\chi$ of $X$,\footnote{A \emph{horizontal character} on $X=N/\Gamma$ is a continuous group homomorphism $\chi$ from $N$ to $\mathbb{T}$ such that $\chi(\Gamma)=1$.} the complement of the set $$A_{\chi}:=\{\alpha\in\mathbb{R}^{m}\colon \chi(b_{\alpha})\neq 1\}$$
	is contained in $W(\X)$. So $\nu(A_{\chi})=1$.
	
%	 has $\nu$-measure equals to 1.
%	Let $e_{j}\in\mathbb{R}^{m}$ denote the vector whose $j$-th coordinate equals to 1 and all other coordinates equal to 0. Suppose that $\chi(T_{te_{j}})=e^{2\pi i c_{j}t}$ for $1\leq j\leq m$ and $t\in\mathbb{R}$. Since $\X$ is ergodic and $\chi$ is non-trivial, at least one of $c_{j}, 1\leq j\leq m$ is nonzero.
%	Then 
%	$A_{\chi}=\{\alpha\in\mathbb{R}^{m}\colon (c_{1},\dots,c_{m})\cdot \alpha\notin\mathbb{Z}\}$
%	and it is not hard to see that 
%	

	%	In fact, for all $\alpha=(\alpha_{1},\alpha_{2})\in\mathbb{R}^{m}$, $\chi(T_{\alpha})=\chi(T_{(1,0)})^{\alpha_{1}}\chi(T_{(0,1)})^{\alpha_{2}}$. Suppose that $\chi(T_{\alpha})= c$. Writing $\chi(T_{(1,0)})=e^{2\pi i c_{1}}$, $\chi(T_{(0,1)})=e^{2\pi i c_{2}}$ and $c=e^{2\pi i c_{0}}$, we have that $c_{1}\alpha_{1}+c_{2}\alpha_{2}=c_{0}+k$ for some $k\in\mathbb{Z}$. It suffices to show that not all of $c_{1}$ and $c_{2}$ equals to 0. If $c_{1}=c_{2}=0$, then $\chi(T_{\alpha})=1$ for all $\alpha\in\mathbb{R}^{m}$, which contradict the the fact that $X$ is ergodic. This proves the claim.
	
	Let $$A=\{\alpha\in\mathbb{R}^{m}\colon \chi(b_{\alpha})\neq 1 \text{ for all nontrivial horizontal character } \chi\}.$$ 
	Since there are only countably many horizontal characters, %the complement of $A$ is contained in the union of at most countably many hyperplanes. So 
	$\nu(A)=1$.
	
	Fix $\alpha\in A$.
	Let $B_{\alpha}$ denote the set of $\beta$ such that $((b_{t\alpha}\Gamma,b_{t\beta}\Gamma))_{t\in\mathbb{R}}$ is not equidistributed on $X\times X$.
	Then for $\beta\in B_{\alpha}$, by Leibman's Theorem \cite{L}, there exists a nontrivial horizontal character $\chi_{\alpha,\beta}$ of $X\times X$ such that $\chi_{\alpha,\beta}(b_{\alpha},b_{\beta})=1$. Since $\alpha\in A$, there exists horizontal characters $\chi$ and $\chi'$ of $X$ such that $\chi(b_{\alpha})=\chi'(b_{\beta})\neq 1$.
	
	For all horizontal characters $\chi$ and $\chi'$ of $X$, let 
	$$B_{\chi,\chi'}=\{\gamma\in\mathbb{R}^{m}\colon\chi(b_{\alpha})=\chi'(b_{\gamma}) \}.$$
	$B_{\chi,\chi'}$ is obviously non-empty. Pick any 
	$\gamma_{0}\in B_{\chi,\chi'}$. Then $$B_{\chi,\chi'}=\{\gamma\in\mathbb{R}^{m}\colon \chi'(b_{\gamma-\gamma_{0}})=1\}$$
	which is contained in $W(\X)+\gamma_{0}$. By assumption, $\nu(B_{\chi,\chi'})=0$.

	Since $B_{\alpha}=\cup_{\chi\not\equiv 1}\cup_{\chi'}B_{\chi,\chi'}$, we have that $\nu(B_{\alpha})=0$. So $\nu\times\nu(J^{c})\leq \nu\times\nu(\{(\alpha,\beta)\colon \alpha\in A, \beta\in B_{\alpha}\})=0$. This finishes the proof. 
\end{proof}

\begin{proof}[Proof of Theorem \ref{main3}]
	We start with the first part. Let $\X=(X,\mathcal{X},\mu, (T_{\alpha})_{\alpha\in \mathbb{R}^{m}})$ be an  ergodic $\mathbb{R}^{m}$-system which is good for double Birkhoff averages such that $\nu(W(\X)+\beta)=0$ for all $\beta\in\mathbb{R}^{m}$. Since $\X$ is separable, it suffices to show that 
	for all $f\in L^{\infty}(\mu)$, there exists $A\subseteq\mathbb{R}$ of density 1 such that
	\begin{equation}\label{31}
	\begin{split}
	\lim_{t\in A,t\to\infty}\int_{\mathbb{R}^{m}}f(T_{t \alpha}x)\,d\nu(\alpha)=\int_{X}f\,d\mu
	\end{split}
	\end{equation}
	for $\mu$-a.e. $x\in X$. Suppose first that $f$ is measurable with respect to $Z_{\mathbb{R}^{m},\mathbb{R}^{m}}$. Consider the factor system  $\Y=(X,Z_{\mathbb{R}^{m},\mathbb{R}^{m}},\mu, (T_{\alpha})_{\alpha\in \mathbb{R}^{m}})$ of $\X$. Since $W(\Y)\subseteq W(\X)$, we have that $\nu(W(\Y)+\beta)=0$ for all $\beta\in\mathbb{R}^{m}$. 
	So (\ref{31}) follows from Propositions \ref{main21}.
	
	 We now assume that $\mathbb{E}(f\vert Z_{\mathbb{R}^{m},\mathbb{R}^{m}})=0$. To show (\ref{31}),
	it  suffices  to show that 
	\begin{equation}\nonumber
	\begin{split}
	\lim_{T\to\infty}\frac{1}{T}\int_{0}^{T}\vert\int_{\mathbb{R}^{m}}f(T_{t\alpha}x)\,d\nu(\alpha)\vert^{2}\,dt=0
	\end{split}
	\end{equation}
	for $\mu$-a.e. $x\in X$. 
	By Proposition \ref{Hf} and the assumption that $\X$ is good for double Birkhoff averages,
	\begin{equation}\nonumber
	\begin{split}
	&\quad \lim_{T\to\infty}\frac{1}{T}\int_{0}^{T}\vert\int_{\mathbb{R}^{m}}f(T_{t\alpha}x)\,d\nu(\alpha)\vert^{2}\,dt
	\\&=\lim_{T\to\infty}\frac{1}{T}\int_{0}^{T}\int_{\mathbb{R}^{2m}}f(T_{t\alpha}x)\overline{f}(T_{t\beta}x)\,d\nu(\alpha)d\nu(\beta)\,dt
	\\&=\int_{\mathbb{R}^{2m}}\Bigl(\lim_{T\to\infty}\frac{1}{T}\int_{0}^{T}f(T_{t\alpha}x)\overline{f}(T_{t\beta}x)\,dt\Bigr)\,d\nu(\alpha)d\nu(\beta)
	\\&=\int_{\mathbb{R}^{2m}}\Bigl(\lim_{T\to\infty}\frac{1}{T}\int_{0}^{T}\mathbb{E}(f\vert Z_{\alpha,\alpha-\beta})(T_{t\alpha}x)\mathbb{E}(\overline{f}\vert Z_{\beta,\alpha-\beta})(T_{t\beta}x)\,dt\Bigr)\,d\nu(\alpha)d\nu(\beta).
	\end{split}
	\end{equation}
	Let $J$ denote the set of all $(\alpha,\beta)\in\mathbb{R}^{2m}$ such that $Z_{\alpha,\alpha-\beta}=Z_{\beta,\alpha-\beta}=Z_{\mathbb{R}^{m},\mathbb{R}^{m}}$. By Lemma \ref{m0}, $\nu\times\nu(J)=1$.
	So
	\begin{equation}\nonumber
	\begin{split}
	&\quad \int_{\mathbb{R}^{2m}}\Bigl(\lim_{T\to\infty}\frac{1}{T}\int_{0}^{T}\mathbb{E}(f\vert Z_{\alpha,\alpha-\beta})(T_{t\alpha}x)\mathbb{E}(\overline{f}\vert Z_{\beta,\alpha-\beta})(T_{t\beta}x)\,dt\Bigr)\,d\nu(\alpha)d\nu(\beta)
	\\&=\int_{J}\Bigl(\lim_{T\to\infty}\frac{1}{T}\int_{0}^{T}\mathbb{E}(f\vert Z_{\mathbb{R}^{m},\mathbb{R}^{m}})(T_{t\alpha}x)\mathbb{E}(\overline{f}\vert Z_{\mathbb{R}^{m},\mathbb{R}^{m}})(T_{t\beta}x)\,dt\Bigr)\,d\nu(\alpha)d\nu(\beta)=0
	\end{split}
	\end{equation}
	and we are done.

\

%We now prove the "only if" part. Let $\X=(X,\mathcal{X},\mu,(T_{\a})_{\a\in\mathbb{R}^{m}})$ be an ergodic $\mathbb{R}^{m}$-system and suppose that $\nu(W(\X)+\beta)>0$ for some $\beta\in\mathbb{R}^{m}$. We wish to show that $\nu$ is not weakly equidistributed for $\X$. We first claim that it suffices to prove the theorem for the case when $\nu(\{\beta\})>0$ for some $\beta\in\mathbb{R}^{m}$. 
%By Theorem \ref{geo}, there exists $\beta\in\mathbb{R}^{m}$ and a subspace $V$ of $\mathbb{R}^{m}$ contained in $W(\X)$ such that $\nu(V+\beta)>0$ and $\nu(V'+\beta)=0$ for every proper subspace $V'$ of $V$. Let $U$ be the subspace of $\mathbb{R}^{m}$ which is the orthogonal complement of $V$, and let $\pi\colon \mathbb{R}^{m}\to U$ be the natural projection.
%Let $\Y=(X,I(V),\mu,(T_{\a})_{\a\in U})$. Since  for all $V$-invariant function $f$, we have  
%$$\int_{\mathbb{R}^{m}}f(T_{t \alpha}x)\,d\nu(\alpha)=\int_{\mathbb{R}^{m}}f(T_{t \pi(\alpha)}x)\,d\nu(\alpha)=\int_{U}f(S_{t\alpha}x)\,d\pi_{\ast}\nu(\alpha),$$
%where $\pi_{\ast}\nu$ is the push-forward of $\nu$ under $\pi$. In order to show that $\nu$ is not weakly equidistributed for the ergodic $\mathbb{R}^{m}$-system $\X$, it suffices to show that $\pi_{\ast}\nu$ is not weakly equidistributed for the ergodic $U$-system $\Y$. Since $\pi_{\ast}\nu(\{\pi(\beta)\})=\nu(V+\beta)>0$, the claim is proved.
%
%For now on we assume that $\nu(\{\beta\})>0$ for some $\beta\in\mathbb{R}^{m}$. 

We now prove the second part. Let $\X=(X,\mathcal{X},\mu,(T_{\a})_{\a\in\mathbb{R}^{m}})$ be an ergodic $\mathbb{R}^{m}$-system and suppose that $\nu(W(\X)+\beta)>0$ for some $\beta\in\mathbb{R}^{m}$. We wish to show that $\nu$ is not weakly equidistributed for $\X$. 
By Theorem \ref{geo}, there exists $\beta\in\mathbb{R}^{m}$ and a subspace $V$ of $\mathbb{R}^{m}$ contained in $W(\X)$ such that $\nu(V+\beta)>0$ and $\nu(V'+\beta)=0$ for every proper subspace $V'$ of $V$. Again by Theorem \ref{geo}, it is not hard to show that there exists an $(m-1)$-dimensional subspace $V_{0}$ of $\mathbb{R}^{m}$ which contains $V$ such that for every subspace $V''$ of $\mathbb{R}^{m}$ which is contained in $W(\X)$ but not contained in $V$, we have that $V''+V_{0}=\mathbb{R}^{m}$.  Let $U$ be the 1-dimensional subspace of $\mathbb{R}^{m}$ which is the orthogonal complement of $V_{0}$, and let $\pi\colon \mathbb{R}^{m}\to U$ be the natural projection.
Let $\Y=(X,I(V_{0}),\mu,(T_{\a})_{\a\in U})$. Since  for all $V_{0}$-invariant function $f$, we have  
$$\int_{\mathbb{R}^{m}}f(T_{t \alpha}x)\,d\nu(\alpha)=\int_{\mathbb{R}^{m}}f(T_{t \pi(\alpha)}x)\,d\nu(\alpha)=\int_{U}f(T_{t\alpha}x)\,d\pi_{\ast}\nu(\alpha),$$
where $\pi_{\ast}\nu$ is the push-forward of $\nu$ under $\pi$. In order to show that $\nu$ is not weakly equidistributed for the ergodic $\mathbb{R}^{m}$-system $\X$, it suffices to show that $\pi_{\ast}\nu$ is not weakly equidistributed for the ergodic $U$-system $\Y$. 

We may decompose $\pi_{\ast}\nu$ as the sum of two (unnormalized) measures $\pi_{\ast}\nu=\nu_{c}+\nu_{d}$, where $\nu_{c}(\{\beta\})=0$ for all $\beta\in U$, and $\nu_{d}$ is supported on at most countably many points on $U$. Since $\pi_{\ast}\nu(\{\pi(\beta)\})=\nu(V_{0}+\beta)\geq \nu(V+\beta)>0$, we have that $\nu_{d}\neq 0$.
Since $U$ is isomorphic to $\mathbb{R}$, applying the conclusion of the first part, we have that (the normalization of) $\nu_{c}$ is weakly equidistributed for $\Y$. So it suffices to show that (the normalization of) $\nu_{d}$ is not weakly equidistributed for $\Y$.

Suppose that (the normalization of) $\nu_{d}$ is weakly equidistributed for $\Y$.
We may assume that $\nu_{d}=\sum_{j\in J}c_{j}\delta_{\a_{j}}$ for some nonempty countable index set $J$, $c_{j}>0$, $\a_{j}\in U$, where $\a_{j}\neq \a_{j'}$ for $j\neq j'$.  We assume without loss of generality that $\sum_{j\in J}c_{j}=1$. Then for all $f\in L^{\infty}(\mu)$ with $\int_{X}f\,d\mu=0$,
$$\lim_{T\to\infty}\frac{1}{T}\int_{0}^{T}\vert\int_{\mathbb{R}^{m}}f(T_{t \alpha}x)\,d\nu(\alpha)\vert^{2}\,dt=\lim_{T\to\infty}\frac{1}{T}\int_{0}^{T}\vert \sum_{j\in J}c_{j}f(T_{t\a_{j}}x)\vert^{2}\,dt=0$$
for $\mu$-a.e. $x\in X$. Since $f\in L^{\infty}(\mu)$, we have that the $L^{1}(\mu)$ limit of $$\lim_{T\to\infty}\frac{1}{T}\int_{0}^{T}\vert \sum_{j\in J}c_{j}f(T_{t\a_{j}}x)\vert^{2}\,dt$$ also equals to 0.
Since $U$ is of dimension 1, by Theorem \ref{geo}, for all $\a\in U\backslash\{0\}$, $(T_{t\a})_{t\in\mathbb{R}}$ is ergodic for $\Y$. So 
$$\lim_{T\to\infty}\frac{1}{T}\int_{0}^{T}(\int_{X}f(T_{t\a}x)\overline{f}(x)\,d\mu)\,dt=\bold{1}_{\a=0}\cdot\Vert f\Vert^{2}_{L^{2}(\mu)}.$$
So,
\begin{equation}\nonumber
\begin{split}
&\quad \int_{X}\lim_{T\to\infty}\frac{1}{T}\int_{0}^{T}\vert \sum_{j\in J}c_{j}f(T_{t\a_{j}}x)\vert^{2}\,dt\,d\mu(x)
\\&=\sum_{j,j'\in J}c_{j}c'_{j}\lim_{T\to\infty}\frac{1}{T}\int_{0}^{T}(\int_{X}f(T_{t\a_{j}}x)\overline{f}(T_{t\a_{j'}}x)\,d\mu)\,dt
\\&=\sum_{j,j'\in J}c_{j}c'_{j}\lim_{T\to\infty}\frac{1}{T}\int_{0}^{T}(\int_{X}f(T_{t(\a_{j}-\a'_{j})}x)\overline{f}(x)\,d\mu)\,dt
\\&=\sum_{j\in J}c^{2}_{j}\cdot\Vert f\Vert^{2}_{L^{2}(\mu)}>0
\end{split}
\end{equation}
whenever $\Vert f\Vert_{L^{2}(\mu)}>0$ (since $J$ is nonempty), a contradiction. This proves the second part of the theorem.
\end{proof}

We are now ready to prove Theorem \ref{main}.
\begin{proof}[Proof of Theorem \ref{main}] 
	Suppose first that $\nu$ is a Borel measure on $\mathbb{R}^{m}$  such that $\nu(\ell)=0$ for all hyperplanes $\ell$ of $\mathbb{R}^{m}$. Let $\X$ be an ergodic $\mathbb{R}^{m}$-system which is good for double Birkhoff ergodic averages. By Theorem \ref{geo}, for every $\beta\in\mathbb{R}^{m}$, $W(\X)+\beta$ is contained in an at most countable union of hyperplanes of $\mathbb{R}^{m}$. So $\nu(W(\X)+\beta)=0$. By Theorem \ref{main3}, $\nu$ is weakly equidistributed for $\X$. This proofs the ``if" part.

	We now prove the ``only if" part. Suppose that there exists a hyperplane of $\mathbb{R}^{m}$   $$\ell=\{\alpha\in\mathbb{R}^{m}\colon \alpha\cdot\beta=c\}$$ such that $\nu(\ell)\neq 0$, where $\beta\in\mathbb{R}^{m}$ and $c\in\mathbb{R}$. 
	Let $(X,\mathcal{X},\mu)$ be the 1 dimensional torus. 
	Let $(S_{s})_{s\in\mathbb{R}}$ be the ergodic $\mathbb{R}$-action on $X$ given by $S_{s}x=x+s \mod 1$, $x\in[0,1)$. We now consider the $\mathbb{R}^{m}$-system $\X_{0}=(X,\mathcal{X},\mu,(T_{\alpha})_{\alpha\in\mathbb{R}^{m}})$, where $T_{\alpha}=S_{\pi(\alpha)}$ for all $\alpha\in\mathbb{R}^{m}$ with $\pi\colon\mathbb{R}^{m}\to\mathbb{R}$ being the linear map given by $\pi(\alpha)=\alpha\cdot\beta$, $\alpha\in\mathbb{R}^{m}$. This system is obviously good for Birkhoff double averages.

	Note that $W(\X_{0})=\{\alpha\in\mathbb{R}^{m}\colon \alpha\cdot\beta=0\}$ and so $\nu(\ell)=\nu(W(\X_{0})+c)=0$. By Theorem \ref{main3},  $\nu$ is not weakly equidistributed for $\X_{0}$.
\end{proof}

\section{Systems good for double Birkhoff averages}\label{s21}
In this section, we discuss to what extend do the main theorems of this paper apply, i.e. which systems are good for double Birkhoff averages. 
Using an argument similar to the proof of Proposition \ref{Hf} (or use Theorem 3.1 of \cite{BLM}), it is not hard to show:
\begin{lem}\label{18}
	Let $\X=(X,\mathcal{X},\mu,(T_{\alpha})_{\alpha\in \mathbb{R}^{m}})$ be an $\mathbb{R}^{m}$-system. If for all $f_{1},f_{2}\in L^{\infty}(\mu)$ and $\alpha_{1},\alpha_{2}\in\mathbb{R}^{m}$, the limit
	\begin{equation}\label{19}
	\begin{split}
	\lim_{N\to\infty}\frac{1}{N}\sum_{n=0}^{N-1}f_{1}(T_{\alpha_{1}n}x)f_{2}(T_{\alpha_{2}n}x)
	\end{split}
	\end{equation}
	exists for $\mu$-a.e. $x\in X$, then $\X$ is good for double Birkhoff averages.
\end{lem} 

Combining Lemma \ref{18} with past results in literature, we have that 
the following $\mathbb{R}^{m}$-systems $\X=(X,\mathcal{X},\mu,(T_{\alpha})_{\alpha\in \mathbb{R}^{m}})$ are good for double Birkhoff averages:
\begin{itemize}
	\item Assani \cite{Ass}: every weakly mixing $\mathbb{R}^{m}$-system $\X$ such that for every $g\in\mathbb{R}^{m}$, the restriction of $T_{g}$ to the Pinsker algebra of $\X$ (the maximal sub $\sigma$-algebra on which  $T_{g}$ has zero entropy) has singular spectrum with respect to the Lebesgue measure;\footnote{\cite{Ass} only covered the case when $m=1$, but the general case can be deduced by a similar argument combined with results in \cite{H0}. It is worth noting that the results in \cite{Ass} was recently improved by Gutman, Huang, Shao and Ye \cite{G}.} 
	\item Bourgain \cite{B}: every $\mathbb{R}$-system (or every $\mathbb{R}^{m}$-system $\X=(X,\mathcal{X},\mu,(T_{\alpha})_{\alpha\in \mathbb{R}^{m}})$ for which there exist an $\mathbb{R}$-action $(S_{t})_{t\in\mathbb{R}}$ on $\X$ and $\beta\in\mathbb{R}^{m}$ such that $T_{\alpha}=S_{\alpha\cdot\beta}$ for all $\alpha\in\mathbb{R}^{m}$).
	\item Donoso and Sun \cite{DS}: every distal  $\mathbb{R}^{m}$-system $\X$.\footnote{The case $m=1$ was proved by Huang, Shao and Ye \cite{HSY}.}
\end{itemize}

For completeness we recall the definition of distal systems (and refer the readers to \cite{Glas} Chapter 10 for further details).
Let $G$ be a group and $\pi\colon \X=(X,\mathcal{X},\mu,(T_{g})_{g\in G})\to \Y=(Y,\mathcal{Y},\nu,(S_{g})_{g\in G})$ be a factor map between two $G$-systems. We say $\pi$ is an \emph{isometric extension} (or $\X$ is an \emph{isometric extension} of $\Y$) if there exist a compact group $H$, a closed subgroup $\Gamma$ of $H$, and a cocycle $\rho\colon G \times Y\to H$ such that $(X,\mathcal{X},\mu,(T_{g})_{g\in G})\cong (Y\times H/\Gamma,\mathcal{Y}\times\mathcal{H}, \nu\times m, (T_{g})_{g\in G})$, where $m$ is the Haar measure on $H/\Gamma$, $\mathcal{H}$ is the Borel $\sigma$-algebra on $H/\Gamma$, and that for all $g \in G$ and $(y,a\Gamma)\in Y\times H/\Gamma$, we have
$$T_{g}(y,a\Gamma)=(S_{g}y, \rho(g,y)a\Gamma).$$

\begin{defn}\label{dis}
	Let $\X$ be a $G$-system. We say that $\X$ is \emph{distal} if there exist a countable ordinal $\eta$ and a directed family of factors $\X_{\theta}, \theta\leq\eta$ of $\X$ such that
	\begin{enumerate}
		\item  $\X_{0}$ is the trivial system, and $\X_{\eta}=\X$;
		\item  For $\theta<\eta$, the extension $\pi_{\theta}\colon \X_{\theta+1}\to \X_{\theta}$ is isometric and is not an isomorphism;
		\item  For a limit ordinal $\l\leq \eta$, $\X_{\l}=\lim\limits_{\leftarrow \theta<\l} \X_{\theta}$.
	\end{enumerate}
\end{defn}

As a result of \cite{DS}, we have the following applications of Theorems \ref{main} and \ref{main3}:
\begin{prop}\label{p1}
	Let $m\in\mathbb{N}$ and $\nu$ be a Borel measure on $\mathbb{R}^{m}$.
	Then (i) $\nu$ is weakly equidistributed for an ergodic distal $\mathbb{R}^{m}$-system $\X$ if and only if $\nu(\G+\beta)=0$ for every $\beta\in\mathbb{R}^{m}$. (ii) $\nu$ is weakly equidistributed for all ergodic distal $\mathbb{R}^{m}$-systems  if and only if $\nu(\ell)=0$ for all hyperplanes $\ell$ of $\mathbb{R}^{m}$.
\end{prop}
\begin{proof}
	Fix $\alpha_{1},\alpha_{2}\in\mathbb{R}^{m}$ and let $G'$ denote the $\mathbb{Z}$-span of $\alpha_{1},\alpha_{2}$. Then it is easy to see by definition that if $\X=(X,\mathcal{X},\mu,(T_{\alpha})_{\alpha\in\mathbb{R}^{m}})$ is a distal $\mathbb{R}^{m}$-system, then   $(X,\mathcal{X},\mu,(T_{\alpha})_{\alpha\in G'})$	is a distal $\mathbb{Z}^{2}$-system. By \cite{DS}, the limit (\ref{19}) exists for all $f_{1},f_{2}\in L^{\infty}(\mu)$ and $\mu$-a.e. $x\in X$. By Lemma \ref{18}, $\X$ is good for double Birkhoff averages, and so the ``if" parts of (i) and (ii) follow from the first part of Theorem \ref{main3} and the ``if" part  of Theorem \ref{main}, respectively. The ``only if" part of (i) follows from the second part of Theorem \ref{main3}  (which is valid for every $\mathbb{R}^{m}$-system). The ``only if" part of (ii) follows from the ``only if" of Theorem \ref{main} as the system $\X_{0}$ constructed in the proof of Theorem \ref{main} is distal.
\end{proof}

Using the result of \cite{B}, we can deduce Proposition \ref{main4} from Theorem \ref{main3}:
\begin{proof}[Proof of Proposition \ref{main4}]
	Let $\nu$ be a Borel measure on $\mathbb{R}$ and $\X$ be an ergodic $\mathbb{R}$-system. By \cite{B}, $\X$ is good for double Birkhoff averages. Then by Theorem \ref{main3},
	$\nu$ is weakly equidistributed for $\X$ if and only if $\nu(W(\X)+\beta)=0$ for all $\beta\in\mathbb{R}$. Since $\X$ is an ergodic $\mathbb{R}$-system, it is easy to see that $W(\X)=\{0\}$. This finishes the proof.
\end{proof}


\begin{thebibliography}{SSS}
	\bibitem{Ass} I. Assani. \emph{Multiple recurrence and almost sure convergence for weakly mixing dynamical systems}. Israel J. Math.  \textbf{103}  (1998), 111-124. 
	
	\bibitem{A} T. Austin. \emph{Norm convergence of continuous-time polynomial		multiple ergodic averages}. Ergodic Theory and Dynamical Systems. \textbf{32} (2012), no. 2, 361-382.
	
	\bibitem{BL} V. Bergelson and A. Leibman. \emph{Cubic averages and large intersections}. Recent trends in ergodic theory and dynamical systems, 5-19, 
	Contemp. Math., 631, Amer. Math. Soc., Providence, RI, 2015. 
	
	\bibitem{BLM} V. Bergelson, A. Leibman and J. Moreira. \emph{From discrete- to continuous-time	ergodic theorems}. Ergodic Theory and Dynamical Systems. \textbf{32} (2012), no. 2, 383-426.
	
	 
	
	
	\bibitem{BTZ} V. Bergelson, T. Tao and T. Ziegler. \emph{An inverse theorem for the uniformity seminorms associated with the action of $F_{p}^{\infty}$}. Geom. Funct. Anal. \textbf{19} (2010), no. 6, 1539-1596.
	

	
	\bibitem{Bj}  M. Bj\"orklund. \emph{Ergodic theorems for homogeneous dilations}. In: Lenz D., Sobieczky F., Woess W. (eds) Random Walks, Boundaries and Spectra. Progress in Probability, vol 64. (2011) Springer, Basel.

		\bibitem{B}  J. Bourgain. \emph{Double recurrence and almost sure convergence}. J. Reine Angew. Math. \textbf{404} (1990),
		140-161.
	
	\bibitem{CH} 
	J. Chaika and P. Hubert.
	\emph{Circle averages and disjointness in typical flat surfaces on every Teichm\"{u}ller disc}.  
	Bull. Lond. Math. Soc. \textbf{49} (2017), no. 5, 755-769. 
	
	
	
	\bibitem{DS}  S. Donoso and W. Sun. \emph{Pointwise convergence of some multiple ergodic averages}.  Advances in Mathematics, \textbf{330} (2018), 946-996.
	
	
	\bibitem{ET}  M. Einsiedler and T. Ward. \emph{Ergodic theory with a view towards number theory}.  Springer-Verlag London \textbf{259} (2011).
	
	\bibitem{F} H. Furstenberg and Y. Katznelson. \emph{An ergodic Szemer\'edi theorem for commuting transformations}. J. Analyse Math. \textbf{34} (1978), 275-291.
	
	\bibitem{Glas} E. Glasner. \emph{Ergodic theory via joinings}.
	Mathematical Surveys and Monographs, 101. American Mathematical Society, Providence, RI, 2003.
	
	\bibitem{G} Y. Gutman, 
	W. Huang, 
	S. Shao and 
	X Ye. \emph{Almost sure convergence of the multiple ergodic average for certain weakly mixing systems}.
	Acta Mathematica Sinica, English Series. (1) \textbf{34} (2018), 79-90.
	
	
	

	
	\bibitem{H} B. Host. \emph{Ergodic seminorms for commuting transformations and applications}. Studia Math. \textbf{195} (2009), 31-49.
	
		\bibitem{H0} B. Host. \emph{Mixing of all orders and pairwise independent joining}. Israel Journal
		of Mathematics \textbf{76} (1991), 289-298.
	
	
	\bibitem{HK} B. Host and B. Kra. \emph{Nonconventional ergodic averages and nilmanifolds}. Ann. of Math. (2) \textbf{161} (2005), no. 1, 397-488.
	
	\bibitem{HSY}  W. Huang, S. Shao and X. Ye. \emph{Pointwise convergence of multiple ergodic averages and strictly ergodic models}.  J. Analyse Math. to appear, arXiv:1406.5930. 
	
	\bibitem{J} R.L. Jones. \emph{Ergodic averages on spheres}. J. Anal. Math, \textbf{61} (1993), 29-45.

	
	
		\bibitem{KSS}  B. Kra, N. Shah and W. Sun. \emph{Equidistribution of dilated curves on nilmanifolds}. Journal of the London Mathematical Society, doi:10.1112/jlms.12156.
		
		\bibitem{dis2} M.T. Lacey. \emph{Ergodic averages on circles}. J. Anal. Math. \textbf{67} (1995), 199-206.
		
		\bibitem{L} A. Leibman.
		\emph{Pointwise convergence of ergodic averages for polynomial sequences of translations on a nilmanifold}.	 Ergodic Theory Dynam. Systems, {\bf 25} (2005), no. 1, 201-213.
		
		
		
		\bibitem{P} A. Potts. \emph{Multiple ergodic averages for flows and an application}.  Illinois J. Math., {\bf 55} (2011), no. 2, 589-621.
		
		\bibitem{PS}  C. Pugh and M. Shub. \emph{Ergodic elements of ergodic actions}.
		Compositio Math.  \textbf{23}  (1971), 115-122. 
		
		
		
		\bibitem{SY}  N. Shah and P. Yang. \emph{Stretching translates of shrinking curves and Dirichlet's simultaneous approximation}.  arXiv: 1809.05570.
		
		\bibitem{St} E.M. Stein. \emph{Maximal functions: Spherical means}. Proc. Natl. Acad. Sci. U.S.A. \textbf{73} (1976), 2174-2175.
		
		%\bibitem{TZ}  T. Tao and T. Ziegler. \emph{Concatenation Theorems for anti-gowers-uniform functions and Host-Kra characteristic factors}.

		\bibitem{Z}  T. Ziegler. \emph{Nilfactors of $\mathbb{R}^{m}$-actions and configurations in	sets of positive upper density in  $\mathbb{R}^{m}$}. J. Anal. Math., \textbf{99} (2006): 249-266.
\end{thebibliography}
\end{document}